\documentclass[11pt]{amsart}
\usepackage{amssymb,amsthm,amsmath,amstext}
\usepackage{mathrsfs}  
\usepackage{bm}        
\usepackage{mathtools} 

\usepackage[left=1.5in,top=1.2in,right=1.5in,bottom=1.15in]{geometry}

\usepackage[all]{xy}

\usepackage{ulem}
\theoremstyle{plain}
\newtheorem{theorem}{Theorem}[section]
\newtheorem{prop}[theorem]{Proposition}
\newtheorem{lemma}[theorem]{Lemma}
\newtheorem{cor}[theorem]{Corollary}
\newtheorem{conj}[theorem]{Conjecture}

\newtheorem{theoremintro}{Theorem}

\theoremstyle{definition}

\theoremstyle{remark}
\newtheorem{remark}[theorem]{Remark}

\newcommand{\isom}{\cong}

\newcommand{\Z}{\mathbb Z}
\newcommand{\R}{\mathbb R}

\newcommand{\C}{\mathbb C}
\newcommand{\F}{\mathbb F}
\renewcommand{\P}{\mathbb P}
\newcommand{\Q}{\mathbb Q}

\DeclareMathOperator{\Aut}{\mathrm{Aut}}
\DeclareMathOperator{\Br}{\mathrm{Br}}

\DeclareMathOperator{\Pic}{\mathrm{Pic}}

\newcommand{\mult}{^{\times}}
\newcommand{\multwo}{^{\times 2}}

\newcommand{\tensor}{\otimes}
\newcommand{\bslash}{\smallsetminus}
\newcommand{\mapto}[1]{\xrightarrow{#1}}

\newcommand{\mm}{\mathfrak{m}}
\newcommand{\pp}{\mathfrak{p}}

\newcommand{\linedef}[1]{\textsl{#1}}
\newcommand{\et}{\mathrm{\acute{e}t}}
\newcommand{\Het}{H_{\et}}
\newcommand{\ur}{\mathrm{ur}}
\newcommand{\Hur}{H_{\ur}}
\newcommand{\Pfister}[1]{\ll\!{#1}\gg}
\newcommand{\quadform}[1]{\;<\! #1 \!>}

\usepackage[backref=page]{hyperref}
\hypersetup{pdftitle={Failure of the local-global principle for isotropy of quadratic forms over function fields}}
\hypersetup{pdfauthor={Asher Auel, V. Suresh}}
\hypersetup{colorlinks=true,linkcolor=blue,anchorcolor=blue,citecolor=blue}

\begin{document}

\title[Failure of the local-global principle for isotropy]{Failure of the local-global principle for isotropy of quadratic forms over function fields}

\author[Auel and Suresh]{Asher Auel and V.~Suresh}

\address{Asher Auel, Department of Mathematics, Dartmouth College, Kemeny Hall, Hanover, New Hampshire \hfill \texttt{\it E-mail address: \tt asher.auel@dartmouth.edu}\hspace*{9pt}}

\address{V.~Suresh, Department of Mathematics, Emory University,
  Mathematics \& Science Center, Atlanta, Georgia \hfill \texttt{\it E-mail address: \tt suresh.venapally@emory.edu}}




\begin{abstract}
We prove the failure of the local-global principle, with respect to
discrete valuations, for isotropy of quadratic forms in $2^n$
variables over function fields of transcendence degree $n \geq 2$ over
an algebraically closed field of characteristic $\neq 2$.  Our
construction involves the generalized Kummer varieties considered by
Borcea and by Cynk and Hulek as well as new results on the
nontriviality of unramified cohomology of products of elliptic curves
over discretely valued fields.
\end{abstract}

\maketitle

\section*{Introduction}

The Hasse--Minkowski theorem states that if a quadratic form $q$ over
a number field is isotropic over every completion, then $q$ is
isotropic.  This is the first, and most famous, instance of the
\linedef{local-global principle} for isotropy of quadratic forms.
Already for a function field of transcendence degree one over a number
field, Witt~\cite{witt} found examples of the failure of the
local-global principle for isotropy of quadratic forms in 3 (and also 4) variables.  Lind~\cite{lind} and Reichardt~\cite{reichardt}, and
later Cassels~\cite{cassels}, found examples of the failure of the
local-global principle for isotropy of pairs of quadratic forms in 4
variables over $\Q$ (see \cite{aitken_lemmermeyer} for a detailed
account), giving examples of quadratic forms over the function field
$\Q(t)$ by an application of the Amer--Brumer theorem
\cite{amer-brumer}, \cite[Theorem~17.14]{elman_karpenko_merkurjev}.
Cassels, Ellison, and Pfister~\cite{CEP} found examples in 4 variables
over the function field $\R(x,y)$.

Here, we are interested in the failure of the local-global principle
for isotropy of quadratic forms over function fields of higher
transcendence degree over algebraically closed fields.  All our fields
will be assumed to be of characteristic $\neq 2$ and all our quadratic
forms nondegenerate.  A quadratic form is called isotropic if it
admits a nontrivial zero.  If $K$ is a field and $v$ is a discrete
valuation on $K$, we denote by $K_v$ the fraction field of the
completion (with respect to the $v$-adic topology) of the valuation
ring of $v$.  When we speak of the \linedef{local-global principle for isotropy
of quadratic forms}, sometimes referred to as the strong Hasse
principle, in a given dimension $d$ over a given field $K$, we mean
the following statement:
\begin{quote}
If $q$ is a quadratic form in $d$ variables over $K$ and $q$ is
isotropic over $K_v$ for every discrete valuation $v$ on $K$, then $q$
is isotropic over $K$.
\end{quote}
Our main result is the following.

\begin{theoremintro}
\label{thm:main}
The local-global principle for isotropy of quadratic forms fails to
hold in dimension $2^n$ over any function field $K$ of transcendence
degree~$n \geq 2$ over an algebraically closed field $k$ of
characteristic $\neq 2$ other than possibly the algebraic closure of a
finite field.
\end{theoremintro}

Previously, only the case of $n=2$ was known, with the first explicit
examples over $K=\C(x,y)$ appearing in \cite{kim_roush}, and later in
\cite{bevelacqua} and \cite{jaworski}.  For a construction, using
algebraic geometry, over any transcendence degree 2 function field
over an algebraically closed field of characteristic~$0$, see
\cite{auel:oberwolfach}, \cite[\S6]{APS}.  In a previous version of
this work, Theorem~\ref{thm:main} was proved in the case of complex
rational function fields, and left as a conjecture.  Though we no
longer need to make use of it, in \S\ref{sec:general}, we also prove a
``geometric presentation lemma'' of general interest about the
existence of double covers of varieties admitting nontrivial
unramified cohomology in maximal degree, which was conjectured in an
earlier version of this work and was shown to imply
Theorem~\ref{thm:main}.

We recall that by Tsen--Lang theory~\cite[Theorem~6]{lang}, such
function fields are $C_n$-fields, hence have $u$-invariant $2^n$, and
thus all quadratic forms of dimension $> 2^n$ are already isotropic,
thus we provide counterexamples to the local-global principle in the
maximal dimension in which they could occur.

We mention that in the case of transcendence degree $n=1$, where
$K=k(X)$ for a smooth projective curve $X$ over an algebraically
closed field $k$, the local-global principle for isotropy of binary
quadratic forms (the ``global square theorem'') holds when the genus
of $X$ is zero and fails when $X$ has positive genus, see
Remark~\ref{rem:curves}.

Finally, when $k$ is the algebraic closure of a finite field, our
methods no longer work.  Though one can use other techniques to handle
the case of transcendence degree $n=2$ (see
Remark~\ref{rem:surfaces_finite_fields}), proving the failure of the
local-global principle for quadratic forms over function fields $K$ of
transcendence degree $n \geq 3$ over $\overline{\F}_p$ remains an open
problem.  Our method relies on proving the nontriviality of certain
unramified cohomology classes in top degree, see \S\ref{sec:general}.
Already for $n=3$, the existence of threefolds over $\F_p$ or
$\overline{\F}_p$ admitting nontrivial unramified cohomology in degree
3 is an open problem related to the integral Tate conjecture, see
\cite[Question~5.4]{CT_kahn}.

Our result relies on two new ingredients and one very useful trick.
The trick, due to Bogomolov~\cite{bogomolov:trick} and outlined in
\S\ref{sec:trick}, is a kind of refinement of the existence of
transcendence bases, and allows us to reduce the construction of
counterexamples to the local-global principle over general function
fields to the case of rational function fields.  Next, our
construction over rational function fields makes use of so-called
generalized Kummer varieties, first considered by Borcea~\cite{borcea}
and developed by Cynk and Hulek~\cite{cynk_hulek}, which are
constructed as quotients of products of elliptic curves and are
birationally double covers of rational varieties.  Finally, we prove a
new result (Theorem~\ref{thm:nontrivial}) on the nontriviality of
unramified cohomology on products of elliptic curves, which provides
an arithmetic generalization of a result of
Gabber~\cite[Appendice]{colliot:exposant_indice}, see also
Colliot-Th\'el\`ene~\cite{colliot:produit}.

We would like to thank the organizers of the summer school
\textit{ALGAR: Quadratic forms and local-global principles}, at the
University of Antwerp, Belgium, in July 
2017, where the authors
obtained a preliminary version of the result.  We also acknowledge the
Shapiro Visitor Program in the Department of Mathematics at Dartmouth
College, which enabled a prolonged visit, where the authors obtained
the final version of the result.  We would also like to thank Fedor
Bogomolov, Jean-Louis Colliot-Th\'el\`ene, David Leep, and Parimala
for very helpful discussions.  The first author received partial
support from NSA Young Investigator grant H98230-16-1-0321, Simons
Foundation grant 712097, and National Science Foundation grant
DMS-2200845; the second author from National Science Foundation grant
DMS-1463882.

\section{Bogomolov's trick}
\label{sec:trick}

Let $K/k$ be a finitely generated field extension.  Recall that $K/k$
admits a finite \linedef{transcendence basis}, i.e., a set of elements
$x_1, \dotsc, x_n \in K$ that are algebraically independent over $k$
and such that $K/k(x_1,\dotsc,x_n)$ is a finite extension.  
The cardinality of any transcendence basis is equal to the
transcendence degree of $K/k$.

A \linedef{projective model} of $K/k$ is an integral projective
$k$-variety $X$ whose field of rational functions is $k$-isomorphic to
$K$.  By the classical Chow's lemma, every finitely generated field
extension admits a projective model, where the dimension of the model
coincides with the transcendence degree of the extension.

The following statement, a refinement of the existence of
transcendence bases, can be traced back to Bogomolov, in the course of
the proof of \cite[Theorem~1.1]{bogomolov:trick}, cf.\ \cite[Proposition~20]{bogomolov_tschinkel}.  

\begin{lemma}[Bogomolov's trick]
\label{lem:trick}
Let $K/k$ be a finitely generated extension of transcendence degree
$n$.  Assume that $k$ is infinite and that a projective model of $K/k$
admits a smooth $k$-point.  Then for any prime number $p$, there
exists a transcendence basis $x_1, \dotsc, x_n \in K$ such that
$K/k(x_1,\dotsc,x_n)$ is finite of degree prime to~$p$.
\end{lemma} 

We remark that by the Lang--Nishimura theorem, see
\cite{lang:uniformization}, \cite{nishimura} and also
\cite[Proposition~A.6]{RY}, the existence of a smooth $k$-point on a
projective model of $K/k$ implies that any other projective model
admits a $k$-point.  The condition that a projective model admits a
smooth $k$-point also implies that any model is
geometrically integral and generically smooth, see
\cite[\href{https://stacks.math.columbia.edu/tag/0CDW}{Lemma
0CDW}]{stacks-project} and
\cite[\href{https://stacks.math.columbia.edu/tag/0CDW}{Lemma
056V}]{stacks-project}.  In particular, if $k$ is algebraically
closed, then any projective model of $K/k$ admits a smooth $k$-point.

\begin{proof}
As above, since a projective model of $K/k$ is geometrically integral,
it is geometrically reduced, and hence $K/k$ is separably generated by
a result of MacLane, see \cite[Theorem~A1.3]{eisenbud}.  Hence, as in
\cite[Proposition~I.4.9]{hartshorne}, there exists a projective
hypersurface model $X \subset \P^{n+1}$ of $K/k$.  Let $d$ be the
degree of $X$.  If $d=1$, then $X=\P^n$ and there is nothing to prove,
so we can assume that $d>1$.  

Projection from a $k$-point in the complement of $X$ (using that $k$
is infinite) yields a dominant rational map $X \dashrightarrow \P^n$
of degree $d$.  Indeed, it is dominant since the fibers of the
projection are the intersections of $X$ with the lines through the
point, and such intersections are always nonempty, cf.\
\cite[Theorem~I.7.2]{hartshorne}.  Moreover, it is generically finite
of degree $d$ since any line through the point cannot be contained in
$X$, hence must intersect $X$ in a zero-dimensional scheme, which has
length $d$.  Similarly, projection from a smooth $k$-point $P$ of $X$
yields a dominant rational map $X \dashrightarrow \P^n$ of degree
$d-1$.  Indeed, since $P$ is a smooth point, the tangent space to $X$
at $P$ has codimension 1 in $\P^{n+1}$, hence (again using that $k$ is
infinite) the general line in $\P^{n+1}$ through $P$ meets $X$
transversally at $P$ and thus intersects $X$ in a nonempty
zero-dimensional scheme of degree $d$ containing $P$ as an irreducible
component.  Then the general fiber of this projection, which is the
complement of $P$ in the intersection of $X$ with a general line
through $P$, is nonempty (using $d>1$) and has length $d-1$, cf.\
\cite[Example~18.16]{harris}.  Since $d$ and $d-1$ are relatively
prime, no prime number $p$ can divide both, hence the associated
extension of function fields $K = k(X) / k(\P^n) = k(x_1,\dotsc,x_n)$
can be chosen of degree prime to~$p$.
\end{proof}

We remark that the hypothesis on a projective model admitting a
$k$-point is essential.  For example, if $K/k$ is the function field
of a smooth plane conic $X$ with no $k$-point, then there is no
presentation of $K$ as an odd degree extension of a rational function
field $k(x)$. Indeed, $X$ cannot acquire rational points over rational
function fields (see \cite[Lemma~7.15]{elman_karpenko_merkurjev}) or
extensions of odd degree (by Springer's theorem), but does acquire a
rational point over its own function field.

We have the following immediate corollary of Bogomolov's trick.

\begin{cor}
\label{cor:trick}
Let $K$ be a finitely generated field of transcendence degree $n$ over
an algebraically closed field $k$.  Then there exists a
transcendence basis $x_1, \dotsc, x_n \in K$ such that
$K/k(x_1,\dotsc,x_n)$ is of odd degree.
\end{cor}

With this in mind, we now explain how Springer's theorem allows us to
reduce the construction of counterexamples to the local-global
principle for isotropy of quadratic forms over general function fields
to the case of rational function fields.

\begin{prop}
\label{prop:trick}
Let $q$ be a nondegenerate quadratic form over a field $K'$ and let
$K/K'$ be a finite extension of odd degree.  
If $q$ is a counterexample to the local-global principle for isotropy over $K'$, then
$q_{K}$ is such a counterexample over $K$.
\end{prop}
\begin{proof}
By Springer's theorem, since $q$ is anisotropic over $K'$ and $K/K'$
has odd degree, then $q_K$ is anisotropic over $K$.  To show that
$q_{K}$ is locally isotropic over $K$, let $v$ be a discrete valuation
on $K$, which then lies over a discrete valuation $v'$ on~$K'$.  Since
the completion $K_{v}$ is a finite extension of the completion
$K'_{v'}$ and since $q$ is isotropic over $K'_{v'}$, we see that
$q_{K}$ is isotropic over $K_{v}$.
\end{proof}

\section{Unramified cohomology of function fields}
\label{sec:valuations}

We now recall the notion of unramified cohomology, introduced in
\cite{colliot-thelene_ojanguren}, restricting ourselves to mod 2
coefficients.  Readers should consult the excellent survey
\cite{colliot:santa_barbara} for further details.  Let $k$ be a field
of characteristic $\neq 2$ and $K/k$ be a finitely generated
extension.  By a \linedef{discrete valuation} $v$ on $K/k$ we mean a
rank~$1$ discrete valuation~$v$ on~$K$ that is trivial on~$k$.
\smallskip

 For each discrete valuation $v$ on $K/k$ with
residue field $\kappa(v)$, recall the residue map in Galois cohomology
$$
\partial_v : H^n(K,\mu_2^{\tensor n}) \to H^{n-1}(\kappa(v),\mu_2^{\tensor n-1})
$$
which arises from the Gysin sequence associated to the closed point in
the spectrum of the valuation ring $R_v$ of $v$, see
\cite[\S3.3]{colliot:santa_barbara}.  The residue map is uniquely
determined by the property that $\partial_v \bigl( (u_1)\dotsm
(u_{n-1}) \cdot (\pi_v) \bigr) = (\overline{u}_1)\dotsm
(\overline{u}_{n-1})$, where $\pi_v$ is a uniformizer and
$u_1,\dotsc,u_{n-1}$ are units of $R_v$, and $\overline{u}$ means the
image of a unit in $\kappa(v)$.  The degree $n$ unramified cohomology
of $K/k$ is defined by
$$
\Hur^n(K/k,\mu_2^{\tensor n}) = \bigcap_{v}
\ker\bigl(\partial_v : H^n(K,\mu_2^{\tensor n}) \to
H^{n-1}(\kappa(v),\mu_2^{\tensor n-1})\bigr)
$$ 
where the intersection ranges over all discrete valuations $v$ on
$K/k$.  We say that an element $\alpha \in H^n(K,\mu_2^{\tensor n})$
is \linedef{unramified} if it belongs to $\Hur^n(K/k,\mu_2^{\tensor n})$.   

We recall two results about discrete valuations on rational function
fields that will be useful later.

\newpage 
\begin{prop}
\label{prop:valuations}\;~\;
\begin{enumerate}
\item Let $k$ be a field and $K=k(x_1,\dotsc,x_n)$ a rational function
field over $k$ with $n \geq 1$.  For each $1 \leq m \leq n$, there
exists a discrete valuation $v$ on $K/k$ satisfying $v(x_i) = 1$ for
all $1 \leq i \leq m$ and $v(x_i)=0$ for all $m+1\leq i \leq n$.

\item Let $k_0$ be a field with a discrete valuation $v_0$ and residue
field $\kappa_0$.  Then there exists a discrete valuation $v$ on the
rational function field $K_0=k_0(x_1,\dotsc,x_n)$, extending $v_0$ on
$k_0$, and with residue field $\kappa_0(x_1,\dotsc,x_n)$.
\end{enumerate}
\end{prop}  
\begin{proof}
For (a), let $A$ be the localization of $k[x_1, \cdots , x_n]$ at the
prime ideal $(x_1, \cdots , x_m)$.  Then $R = A[y_1, \cdots ,
y_{m-1}]/(x_m - x_1y_1, \cdots , x_m - x_{m-1}y_{m-1})$ is an integral
domain with field of fractions isomorphic to $K$.  Furthermore, the
ideal $\frak{p}$ of $R$ generated by the images of $x_1, \dotsc, x_m$
is a prime ideal and $R_\frak{p}$ is a discrete valuation ring.  The
valuation on $K/k$ given by this discrete valuation ring has the
required properties.  Geometrically, this corresponds to blowing up
the model $\P^n$ of $K/k$ along the linear subspace defined by $x_1=\dotsm=x_m=0$.

For (b), letting $R_0 \subset k_0$ be the valuation ring of $v_0$ and
$\pi_0$ a uniformizer, we take the discrete valuation $v$ on $K_0$
associated to the prime ideal in $R_0[x_1, \dotsc, x_n]
\subset K_0$ generated by $\pi_0$.  By construction,
the residue field of $v$ on $K_0$ is
$\kappa_0(x_1,\dotsc,x_n)$.  Geometrically, this corresponds to
the special fiber of the model $\P^n_{\!R_0}$. 
\end{proof}

\section{Generalized Kummer varieties}
\label{sec:kummer}

In this section, we review a construction, considered in the context
of modular Calabi--Yau varieties \cite[\S2]{cynk_hulek} and
\cite{cynk_schutt}, of a generalized Kummer variety attached to a
product of elliptic curves.  This recovers, in dimension 2, the Kummer
K3 surface associated to a decomposable abelian surface, and in
dimension 3, a class of Calabi--Yau threefolds of CM type considered by
Borcea~\cite[\S3]{borcea}.  We also prove some results about the
unramified cohomology groups in top degree of products of elliptic
curves and their associated generalized Kummer varieties.

Let $E_1, \dotsc, E_n$ be elliptic curves over an algebraically closed
field $k$ of characteristic $\neq 2$ and let $Y=E_1 \times \dotsm
\times E_n$.  Let $\sigma_i$ denote the negation automorphism on $E_i$
and $E_i \to \P^1$ the associated quotient branched double cover.
We extend each $\sigma_i$ to an automorphism of $Y$ by acting
trivially on each $E_j$ for $j \neq i$; the subgroup $G
\subset \Aut(Y)$ they generate is an elementary abelian $2$-group.
Consider the exact sequence of abelian groups
$$
1 \to H \to G \mapto{\Pi} \Z/2 \to 0,
$$ 
where $\Pi$ is defined by sending each $\sigma_i$ to $1$.  Then the
product of the double covers $Y \to {\P^1 \times \dotsm \times \P^1}$
is the quotient by $G$ and we denote by $Y \to X$ the quotient by the
subgroup $H$.  The intermediate quotient $X \to \P^1 \times \dotsm
\times \P^1$ is a double cover, branched over a reducible divisor of
type $(4, \dotsc, 4)$. For $n=2$, this divisor is the union of $4$
vertical fibers and $4$ horizontal fibers of $\P^1 \times \P^1$
meeting in $16$ points.

We point out that $X$ is a singular degeneration of smooth Calabi--Yau
varieties that (geometrically) admits a smooth Calabi--Yau model, see
\cite[Corollary~2.3]{cynk_hulek} and \cite[Section~4]{cynk_schutt}.
For $n=2$, the minimal resolution of $X$ is indeed isomorphic to the
Kummer K3 surface $\text{Kum}(E_1 \times E_2)$.

Given nontrivial classes $\gamma_i \in \Het^1(E_i,\mu_2)$, we consider
the cup product 
\begin{equation}
\label{eq:gamma}
\gamma = \gamma_1 \dotsm \gamma_n \in \Het^n(Y,\mu_2^{\tensor n})
\end{equation}
and its image in $\Hur^n(k(Y)/k,\mu_2^{\tensor n})$ under restriction
to the generic point. These classes have been studied in
\cite{colliot:exposant_indice}.  We remark that $\gamma$ is in the
image of the restriction map $H^n(k(\P^1\times \dotsm \times
\P^1),\mu_2^{\tensor n}) \to H^n(k(Y),\mu_2^{\tensor n})$ in Galois
cohomology since each $\gamma_i$ is in the image of the restriction
map $H^1(k(\P^1),\mu_2) \to H^1(k(E_i),\mu_2)$.

We make this more explicit as follows.  Corresponding to each double cover
$E_i \to \P^1$,  choose a Weierstrass equation in
Legendre form
\begin{equation}
\label{eq:legendre}
y_i^2 = x_i (x_i - 1) (x_i - \lambda_i)
\end{equation}
where $x_i$ is a coordinate on $\P^1$ and
$\lambda_i \in k \bslash \{0,1\}$, see \cite[III.1.7]{silverman}.
Then the branched double cover $X \to \P^1 \times \dotsm \times \P^1$
is birationally defined by the equation
\begin{equation}
\label{eq:y}
y^2 = \prod_{i=1}^n x_i (x_i - 1) (x_i - \lambda_i) = f(x_1,\dotsc,x_n)
\end{equation}
where $y=y_1\dotsm y_n$ in $k(Y)$, see \cite[\S3]{cynk_schutt}.  Up to
an automorphism, we can, and henceforth will, choose the Legendre
forms so that the image of $\gamma_i$ under the map
$\Het^1(E_i,\mu_2) \to H^1(k(E_i),\mu_2)$ coincides with the square
class $(x_i) \in k(E_i)/k(E_i)^{\times 2} = H^1(k(E_i),\mu_2)$ of the
rational function $x_i$, which is then visibly in the image of the
restriction map $H^1(k(\P^1),\mu_2) \to H^1(k(E_i),\mu_2)$.  Hence we
see that the (ramified) cup product class
$\xi = (x_1) \dotsm (x_n) \in H^n(k(x_1,\dotsm,x_n),\mu_2^{\tensor
  n})$, restricts to the unramified class
$\gamma \in \Hur^n(k(Y)/k,\mu_2^{\tensor n})$.

The first main result of this section is that the class $\xi$ already
restricts to an unramified class over the quadratic extension $k(X)$.
We prove a more general result that can be viewed as a higher
dimensional generalization of \cite[\S1]{colliot:santa_barbara}.

\begin{prop}
Let $k$ be an algebraically closed field of characteristic $\neq
2$ and $K=k(x_1,\dotsc,x_n)$ a rational function field over $k$.  For
$1 \leq i \leq n$, let $f_i(x_i) \in k[x_i]$ be polynomials of even
degree satisfying $f_i(0)\neq 0$, and let $f = \prod_{i=1}^n
x_i f_i(x_i)$.  Then the restriction of the class $\xi = (x_1) \dotsm
(x_n) \in H^n(K,\mu_2^{\tensor n})$ to $H^n(K(\sqrt{f}),\mu_2^{\tensor
n})$ is unramified with respect to all discrete valuations.
\end{prop} 
\begin{proof}
Let $L = K(\sqrt{f})$ and $v$ a discrete valuation on $L$ with
valuation ring $R$, maximal ideal $\mm$, and residue field
$\kappa$.  Write $\xi_L$ for the restriction of $\xi$ to $H^n(L,\mu_2^{\tensor n})$.

Suppose $v(x_i) < 0$ for some $i$.  Let $d_i$ be the degree of $f_i$
and consider the reciprocal polynomial $f_i^*(x_i) = x_i^{d_i}
f_i(\frac{1}{x_i})$, so that $x_if_i(x_i) = x_i^{d_i+1}\cdot
\frac{1}{x_i}f_i^*(\frac{1}{x_i})$.  Since $d_i$ is even, we have that
the polynomials $x_if_i(x_i)$ and $\frac{1}{x_i}f_i^*(\frac{1}{x_i})$
have the same class in $K\mult/K\multwo$.  

Thus, up to replacing, for all $i$ with $v(x_i) < 0$, the polynomial
$f_i$ by $f_i^*$ in the definition of $f$ and replacing $x_i$ by
$\frac{1}{x_i}$, we can assume that $v(x_i) \geq 0$ for all $i$
without changing the extension $L/K$.  Hence $k[x_1, \dotsc, x_n]
\subset R_v$.  

Consider $\pp = k[x_1, \dotsc, x_n] \cap \mm$. Then $\pp$ is a prime
ideal of $k[x_1, \dotsc, x_n]$ whose residue field $\kappa(\pp)$ is a
subfield of $\kappa$.  Let $K_\pp$ be the completion of $K$ at $\pp$
and $L_v$ the completion of $L$ at $v$. Then $K_\pp$ is a subfield of
$L_v$.
 
If $v(x_i) = 0$ for all $i$, then $\xi_L$ is unramified at $v$.  So
suppose that $v(x_i) \neq 0$ for some $i$.  By reindexing $x_1,
\dotsc, x_n$, we assume that there exists $m \geq 1$ such that $v(x_i)
> 0$ for $1 \leq i \leq m$ and $v(x_i) = 0$ for $m+1 \leq i \leq n$,
i.e., we have $x_1, \dotsc, x_m \in \pp$ and $x_{m+1}, \dotsc, x_n \not\in
\pp$.  In particular, the transcendence degree of $\kappa(\pp)$ over
$k$ is~${\leq n-m}$.

First, suppose $f_i(x_i) \in \pp$ for some $m+1 \leq i \leq n$. Since
$f_i(x_i)$ is a product of linear factors in $k[x_i]$, we have that
$x_i - a_i \in \pp$ for some $a_i \in k$, with $a_i \neq 0$ since
$f_i(0) \neq 0$. Thus the image of $x_i$ in $\kappa(\pp)$ is equal to $a_i$
and hence is a square in $K_\pp$. In particular, $x_i$ is a square in
$L_v$, thus $\xi_L$ is trivial (hence unramified) at $v$, cf.\ \cite[Proposition~1.4]{colliot-thelene_ojanguren}.

Now, suppose that $f_i(x_i) \not\in \pp$ for all $m+1 \leq i \leq n$.
Then for each $1 \leq i \leq m$, we see that since $x_i \in \pp$ and
$f_i(0) \neq 0$, we have $f_i(x_i) \not\in \pp$.  Consequently, we can
assume that $f = x_1\dotsm x_m u$ for some $u \in k[x_1, \dotsc, x_n]
\bslash \pp$.  We remark that $f = x_1\dotsm x_m u$ is a square in
$L$, so that $(x_1\dotsm x_m) = (u)$ in $H^1(L,\mu_2)$.  

For $m=1$, we see that $\xi_L = (u) \cdot (x_2) \dotsm (x_n)$ is
unramified at $v$ since $u$ and $x_2, \dotsc, x_n$ are units at $v$.

For $m > 1$, a computation with symbols
$$
(x_1) \dotsm (x_m) = (x_1) \dotsm
(x_{m-1}) \cdot (x_1\dotsm x_m) = (x_1) \dotsm (x_{m-1}) \cdot (u) \in H^m(L, \mu_2^{\tensor m})
$$
shows that $\xi_L = (x_1) \dotsm (x_{m-1}) \cdot (u) \cdot (x_{m+1})
\dotsm (x_n)$.  Since $u$ and $x_{m+1},\dotsc,x_n$ are units at $v$,
computing with the Galois cohomology residue homomorphism $\partial_v :
H^n(L, \mu_2^{\tensor n}) \to H^{n-1}(\kappa(v), \mu_2^{\tensor
n-1})$ from \S\ref{sec:valuations} shows that
$$
\partial_v\bigl((x_1) \dotsm (x_{m-1}) \cdot (u) \cdot (x_{m+1}) \dotsm
(x_n) \bigr) = \alpha \cdot (\overline{u}) \cdot (\overline{x}_{m+1}) \dotsm
(\overline{x}_n)
$$ 
for some $\alpha \in H^{m-2}(\kappa(v), \mu_2^{\tensor m-2})$, where
for any $h \in k[x_1, \dotsc, x_n]$, we write $\overline{h}$ for the
image of $h$ in $\kappa(\pp) \subset \kappa$.  Since the transcendence
degree of $\kappa(\pp)$ over $k$ is $\leq n-m$ and $k$ is
algebraically closed, we have that $\kappa(\pp)$ has $2$-cohomological
dimension $\leq n-m$ by
\cite[II.4.2~Proposition~11]{serre:galois_cohomology}, so that
$H^{n-m+1}(\kappa(\pp), \mu_2^{\tensor n-m+1})= 0$. Since
$\overline{u}, \overline{x}_i \in \kappa(\pp)$, we then have that
$(\overline{u}) \cdot (\overline{x}_{m+1}) \dotsm (\overline{x}_n)$ is
trivial.  In particular, $\partial_v(\xi_L)$ is trivial, and hence
$\xi_L$ is unramified at $v$.  Finally, we have shown that the
restriction $\xi_L$ is unramified at all discrete valuations on~$L$.
\end{proof}

As an immediate consequence, we deduce the fact that the class $\xi$
 restricts to an unramified class over $k(X) = k(x_1,\dotsc,
 x_n)(\sqrt{f})$, where $f$ is as in \eqref{eq:y}.

\begin{prop}
\label{prop:main_unram}
Let $E_1, \dotsc, E_n$ be elliptic curves over an algebraically closed
field~$k$ of characteristic $\neq 2$, given in the Legendre form
\eqref{eq:legendre}, with $K = k(x_1,\dotsc,x_n)$. Then the
restriction of the class $\xi = (x_1) \dotsm (x_n)$ in
$H^n(K,\mu_2^{\tensor n})$ to $H^n(k(X),\mu_2^{\tensor n})$ is
unramified at all discrete valuations.
\end{prop}

This unramified class on $k(X)$ restricts to the class $\gamma$ on
$k(Y)$ in \eqref{eq:gamma}, so without loss of generality, we will
also call it $\gamma$.  Finally, we will need conditions ensuring that
our class $\gamma$ is nontrivial over $k(X)$.  For this, we must
choose the elliptic curves $E_1, \dotsc, E_n$ more carefully, and we
will then show that $\gamma$ is nontrivial over $k(Y)$, hence is
nontrivial over $k(X)$.  We proceed as follows.  

First, we choose a subfield $k_0 \subset k$ admitting a discrete
valuation $v_0$.  This is possible unless $k$ is the algebraic closure
of a finite field; this is why we must henceforth assume that $k$ is
not the algebraic closure of a finite field.  Then we choose $E_i$
defined over $k_0$ with Weierstrass equation \eqref{eq:legendre}
satisfying $v_0(\lambda_i) > 0$.  Finally, we appeal to the following
arithmetic version, which was inspired by
Bogomolov~\cite[\S7]{bogomolov:stable}, of a result of
Gabber~\cite[Appendice]{colliot:exposant_indice}.

\begin{theorem}
\label{thm:nontrivial}
Let $k_0$ be a field with a discrete valuation $v_0$ whose residue
field has characteristic $\neq 2$. Let $E_1, \dotsc, E_n$ be elliptic
curves over $k_0$ given in the Legendre form \eqref{eq:legendre}, with
$v_0(\lambda_i) > 0$ for all $1 \leq i \leq n$.  Let
$Y = E_1 \times \dotsm \times E_n$ and $k/k_0$ be an algebraically
closed extension.  Then the class
$\gamma \in H^n(k(Y), \mu_2^{\tensor n})$ in \eqref{eq:gamma} is
nontrivial.
\end{theorem}

\begin{proof}
Let $K_0 = k_0(x_1, \dotsc, x_n)$ and let $\gamma_0$ be the
restriction of the class $\xi_0 = (x_1)\dotsm (x_n) \in H^n(K_0,
\mu_2^{\tensor n})$ to $H^n(k_0(Y), \mu_2^{\tensor n})$.  Letting
$\kappa_0$ be the residue field of $v_0$, by
Proposition~\ref{prop:valuations}\textit{b)} we can extend $v_0$ to a
discrete valuation on $K_0$ with residue field $\kappa_0(x_1,\dotsc,
x_n)$.  We remark that each $x_i \in K_0$ is a unit with respect to
this valuation.  Since $k_0(Y)/K_0$ is a finite separable extension,
we can further extend this valuation to a discrete valuation
$\tilde{v}$ on $k_0(Y)$. Writing
$$
k_0(Y) = k_0(x_1, \cdots,x_n)(\sqrt{x_1(x_1 - 1)(x_1 - \lambda_1)},
\cdots , \sqrt{x_n(x_n - 1)(x_n - \lambda_n)})
$$
then since $\tilde{v}(\lambda_i) > 0$ and $\tilde{v}(x_i) = 0$ for
all $i$, we have that the residue field of $\tilde{v}$ is
$$
\tilde{\kappa} = \kappa_0(x_1, \cdots,x_n)(\sqrt{x_1 - 1}, \cdots,
\sqrt{x_n - 1}).
$$
Since each $x_i$ is a unit at $\tilde{v}$, the
class $\gamma_0$ is unramified at $\tilde{v}$, and has
specialization $\tilde{\xi}_0 = (x_1) \cdots (x_n) \in
H^n(\tilde{\kappa}, \mu_2^{\tensor n})$.

We now argue that $\tilde{\xi}_0$ is nontrivial, hence that $\gamma_0$
is nontrivial.  To this end, by
Proposition~\ref{prop:valuations}\textit{a)} there is a valuation
$v_n$ on $\kappa_0(x_1, \cdots,x_{n})$ such that $v_n(x_i) = 0$ for $1
\leq i \leq n-1$ and $v_n(x_n)=1$, and we denote by $\tilde{v}_n$ an
extension to $\tilde{\kappa}$, which is separable over
$\kappa_0(x_1,\dotsc,x_n)$ and unramified at $\tilde{v}_n$.  Thus
$\tilde{v}_n$ is trivial on the subfield
$$
\tilde{\kappa}_n = \kappa_0(x_1, \cdots,x_{n-1})(\sqrt{x_1-1}, \cdots,
\sqrt{x_{n-1}-1})
$$
and satisfies $\tilde{v}_n(x_n)=1$.  Then the residue field of
$\tilde{v}_n$ is $\tilde{\kappa}_n(\sqrt{-1})$ and the residue of the
class $\tilde{\xi}_0$ at $\tilde{v}_n$ is simply $(x_1) \cdots
(x_{n-1})$.  Repeatedly taking residues using this process, we arrive
at the class $(x_1) \in H^1(\kappa_0(x_1)(\sqrt{-1},\sqrt{x_1-1}),\mu_2)$, which is nontrivial, hence $\tilde{\xi}_0$
is nontrivial.  Thus $\gamma_0 \in H^n(k_0(Y),\mu_2^{\tensor n})$ is
nontrivial.

Now let $k/k_0$ be any algebraically closed field extension and let
$\bar{k}_0$ be the algebraic closure of $k_0$ in $k$.  First, we show
that the restriction of $\gamma_0$ to $H^n(\bar{k}_0(Y),\mu_2^{\tensor
n})$ is nontrivial.  This is equivalent to the restriction of
$\gamma_0$ to $H^n(l_0(Y),\mu_2^{\tensor n})$ being nontrivial for
every finite algebraic extension $l_0/k_0$.  Letting $w_0$ be an
extension of $v_0$ to $l_0$, we still have that $w_0(\lambda_i) > 0$
for all $i$, so we can apply what we have already proved.  Second,
since $\gamma_0$ is unramified, its restriction to
$H^n(\bar{k}_0(Y),\mu_2^{\tensor n})$ and further to
$H^n(k(Y),\mu_2^{\tensor n})$, remains unramified and coincides with
the class $\gamma$.  Then we can appeal to the rigidity property for
unramified cohomology, which implies that the restriction map
$\Hur^n(\bar{k}_0(Y)/\bar{k},\mu_2^{\tensor n}) \to
\Hur^n(k(Y)/k,\mu_2^{\tensor n})$ is an isomorphism, see
\cite[\S4.4]{colliot:santa_barbara}, showing that $\gamma$ is
nontrivial.
\end{proof}

Additional aspects and applications of the argument in the proof of
Theorem~\ref{thm:nontrivial} will be the subject of forthcoming work
\cite{auel_suresh:elliptic_unramified}.  In particular,
$\mu_2^{\tensor n}$ coefficients can be replaced by $\mu_\ell^{\tensor
n}$ coefficients for any positive integer $\ell$ prime to the residue
characteristic of $k_0$.  We content ourselves with giving one
application here, which is a new proof of (a generalization of)
Gabber's result \cite[Appendice]{colliot:exposant_indice}.

\begin{cor}
Let $k$ be a field of characteristic $\neq 2$ and $K/k$ an
algebraically closed extension.  Let $E_1, \dotsc, E_n$ be elliptic
curves over $K$ whose $j$-invariants are algebraically independent
over $k$.  Let $Y = E_1 \times \dotsm \times E_n$.  Then the class
$\gamma \in H^n(K(Y), \mu_2^{\tensor n})$ in \eqref{eq:gamma} is
nontrivial.
\end{cor}

\begin{proof}
Since $K$ is algebraically closed, each elliptic curve $E_i$ can be
put into Legendre form \eqref{eq:legendre}. Hence $Y$ is defined over
the field $k_0 = k(\lambda_1, \dotsc, \lambda_n)$.  Since 
the $j$-invariant of $E_i$ is a rational function in $\lambda_i$, the
algebraic independence of $j(E_1), \dotsc, j(E_n)$ over~$k$ implies
the algebraic independence of $\lambda_1, \dotsc, \lambda_n$ over~$k$.
By Proposition~\ref{prop:valuations}\textit{a)}, there exists a
discrete valuation $v_0$ on $k_0$ such that $v_0(\lambda_i)>0$ for all
$i$, and then we can apply Theorem~\ref{thm:nontrivial}.
\end{proof}

\section{Hyperbolicity over a quadratic extension}
\label{sec:pfister}

Let $K$ be a field of characteristic $\neq 2$. We will need the
following result about isotropy of quadratic forms, generalizing a
well-known result in the dimension four case, see \cite[Ch.~2,~Lemma~14.2]{scharlau:book}.

\begin{prop}
\label{prop:overL}
Let $q$ be a quadratic form over $K$ of dimension divisible by $4$
and discriminant $d$, and let $L = K(\sqrt{d})$.  If $q$
is hyperbolic over $L$ then $q$ is isotropic over $K$.
\end{prop}
\begin{proof}
If $d \in K\multwo$, then $K = L$ and there is nothing to prove, so
suppose {$d \not\in K\multwo$}.  To get a contradiction, we will
assume $q$ is anisotropic.  Since $q_L$ is hyperbolic, we then have $q
\simeq \, \quadform{1, -d} \tensor\; q_1$ for some quadratic form
$q_1$ over $K$, see \cite[Ch.~2,~Theorem~5.2]{scharlau:book}.  Since
the dimension of $q$ is divisible by four, the dimension of $q_1$ is
divisible by two, and a computation of the discriminant shows that $d
\in K\multwo$, which is a contradiction.
\end{proof}

For $n \geq 1$ and $a_1, \dotsc, a_n \in K\mult$, recall the $n$-fold
Pfister form
$$
\Pfister{a_1, \dotsc, a_n} \;= \quadform{1, -a_1}\, \tensor \dotsm
\tensor \quadform{1, -a_n}
$$
and the associated symbol $(a_1) \dotsm (a_n)$ in the Galois
cohomology group $H^n(K,\mu_2^{\tensor n})$.  Then $\Pfister{a_1,
\dotsc, a_n}$ is hyperbolic if and only if $\Pfister{a_1, \dotsc,
a_n}$ is isotropic if and only if $(a_1) \dotsm (a_n)$ is trivial.
For the fact that isotropic Pfister forms are hyperbolic, see
\cite[Ch.~4,~Corollary~1.5]{scharlau:book}.  The fact that the
triviality of $(a_1) \dotsm (a_n)$ implies the hyperbolicity of
$\Pfister{a_1, \dotsc, a_n}$ is a consequence of the Milnor
conjectures for the Witt group, as proved by
Voevodsky~\cite{voevodsky:Milnor_conjecture} and Orlov, Vishik,
Voevodsky~\cite{orlov_vishik_voevodsky}.

For $d \in K\mult$ and $n \geq 2$, we will consider quadratic forms of
discriminant $d$ related to $n$-fold Pfister forms, as follows.  Write
$\Pfister{a_1, \dotsc, a_n}$ as $q_0 \perp \quadform{(-1)^n a_1 \dotsc
a_n}$, then define $\Pfister{a_1,\dotsc,a_n;d} \;=\; q_0 \,\perp
\quadform{(-1)^n a_1 \dotsc a_n d}$.  For example:
\begin{align*}
\Pfister{a;d} {} = & \quadform{1, -a d} \\
\Pfister{a,b;d} \;\;= & \quadform{1, -a, -b, abd} \\
\Pfister{a,b,c;d} \;\;= & \quadform{1, -a, -b, -c, ab, ac, bc, -abcd}
\end{align*}
for $n=1, 2, 3$, respectively.  We remark that every quadratic
form of dimension~$4$ is similar to one of this type.  We also remark
that $\Pfister{a_1,\dotsc,a_n;d}$ becomes isomorphic to
$\Pfister{a_1,\dotsc,a_n}$ over $K(\sqrt{d})$.  In general, these
quadratic forms are examples of \linedef{twisted Pfister forms} in the
sense of Hoffmann~\cite{hoffmann}.

\begin{prop}
\label{prop:twisted_split}
  Assume $n \geq 2$.  If $q = \Pfister{a_1,\dotsc,a_n;d}$ and $L = K(\sqrt{d})$ then $q$
  is isotropic if and only if $q_L$ is isotropic if and only if
  $(a_1)\dotsm (a_n) \in H^n(L,\mu_2^{\tensor n})$ is trivial.
\end{prop}
\begin{proof}
If $q$ is isotropic then $q_L$ is isotropic.  If $q_L$ is isotropic,
then as mentioned above, it is hyperbolic as it is a Pfister form,
hence by Proposition~\ref{prop:overL} (since $q$ has dimension $2^n$
and $n \geq 2$), $q$ is isotropic over $K$.  As previously mentioned above
(and consequence of the Milnor conjectures), $(a_1)\dotsm (a_n) \in
H^n(L,\mu_2^{\tensor n})$ is trivial if and only the Pfister form
$q_L$ is isotropic.
\end{proof}

This generalizes a well-known result about quadratic
forms of dimension 4, see \cite[Ch.~2,~Lemma~14.2]{scharlau:book}.

\section{Failure of the local global principle}
\label{sec:unram}

In this section, we prove our main Theorem~\ref{thm:main} by providing
a construction of quadratic forms over function fields that are
locally isotropic yet globally anisotropic.  First we prove a general
result about the generalized Pfister forms in
Section~\ref{sec:pfister}.

\begin{prop}
\label{prop:local_isotropic}
Let $k$ be an algebraically closed field of characteristic $\neq 2$
and $K/k$ a finitely generated extension of transcendence degree $n
\geq 2$.  Let $a_1, \dotsc, a_n, d \in K\mult$ be such that the symbol
$(a_1)\dotsm (a_n)$ in $H^n(K,\mu_2^{\tensor n})$ becomes unramified
over $L=K(\sqrt{d})$.  Then the quadratic form $q =
\Pfister{a_1,\dotsc,a_n;d}$ is locally isotropic over $K$.
\end{prop}
\begin{proof}
Let $v$ be a discrete valuation on $K$ and $w$ an extension to $L$,
with completions $K_v$ and $L_w$ and residue fields $\kappa(v)$ and
$\kappa(w)$, respectively.  By assumption, the restriction of the
symbol $(a_1) \dotsm (a_n)$ to $H^n(L,\mu_2^{\tensor n})$ is
unramified at $w$.  By cohomological purity for discrete valuation
rings (cf.\ \cite[\S3.3]{colliot:santa_barbara}) we have a surjective
map $\Het^n(R_w,\mu_2^{\tensor n}) \to \Hur^n(L_w/k,\mu_2^{\tensor
n})$ where $R_w \subset L_w$ is the valuation ring.  By proper base
change (cf.\ \cite[XII.5.5]{SGA4}, see also a general result of Gabber
\cite[\href{https://stacks.math.columbia.edu/tag/09ZI}{Tag
09ZI}]{stacks-project}), we have an isomorphism
$\Het^n(R_w,\mu_2^{\tensor n}) \isom H^n(\kappa(w),\mu_2^{\tensor
n})$.  Since $\kappa(w)/k$ has transcendence degree $<n$ by
Abhyankar's inequality~\cite[Corollary~1(1)]{abhyankar} and $k$ is
algebraically closed, we have that $\kappa(w)$ has $2$-cohomological
dimension $<n$ by
\cite[II.4.2~Proposition~11]{serre:galois_cohomology}.  From all this,
we deduce that $\Hur^n(L_w/k,\mu_2^{\tensor n})=0$.  In particular,
the symbol $(a_1) \dotsm (a_n)$ has trivial restriction to
$H^n(L_w,\mu_2^{\tensor n})$.  Thus by
Proposition~\ref{prop:twisted_split}, we have that $q_{K_v}$ is
isotropic.  Finally, as this holds for every discrete valuation~$v$
on~$K$, the quadratic form $q$ is locally isotropic over $K$.
\end{proof}

Now, we will utilize our constructions in
Section~\ref{sec:kummer}. Let $k$ be an algebraically closed field of
characteristic $\neq 2$ that is not the algebraic closure of a finite
field.  Let $k_0 \subset k$ be a subfield with a discrete valuation
$v_0$ whose residue field has characteristic~$\neq~2$. Let $E_1,
\dotsc, E_n$ be elliptic curves over $k_0$ given in the Legendre
form~\eqref{eq:legendre}, with $v_0(\lambda_i) > 0$ for all $1 \leq i
\leq n$.  Let $X \to \P^1 \times \dotsm \times \P^1$ be the double
cover defined by $y^2 = \prod_{i=1}^nx_i(x_i-1)(x_i-\lambda_i) =
f(x_1,\dotsc,x_n)$ in \eqref{eq:y}, and consider the quadratic form
\begin{equation}
\label{eq:counter_example}
q = \;\Pfister{x_1,\dotsc,x_n;f}
\end{equation}
over the rational function field $k(\P^1 \times \dotsm \times \P^1) =
k(x_1,\dotsc,x_n)$, as in Section~\ref{sec:pfister}.

Our main result is that for $n \geq 2$, the quadratic form $q$ shows
the failure of the local-global principle for isotropy, with respect
to all discrete valuations, for quadratic forms of dimension $2^n$
over $k(x_1,\dotsc,x_n)$.

\begin{theorem}
\label{thm:locally}
Let $k$ be an algebraically closed field of
characteristic $\neq 2$ that is not the algebraic closure of a finite
field and assume $n \geq 2$.  The quadratic form $q =
\;\Pfister{x_1,\dotsc,x_n;f}$ as in \eqref{eq:counter_example} is anisotropic over $k(x_1,\dotsc,x_n)$
yet is isotropic over the completion at every discrete valuation.
\end{theorem}
\begin{proof}
Write $K = k(x_1,\dotsc,x_n)$ and $L = K(\sqrt{f})$.  By
Proposition~\ref{prop:main_unram}, the restriction of the symbol
$(x_1) \dotsm (x_n) \in H^n(K,\mu_2^{\tensor n})$ to $L$ is
unramified.  Hence Proposition~\ref{prop:local_isotropic} implies that
$q$ is locally isotropic at every discrete valuation on $K$.

The restriction of the symbol $(x_1) \dotsm (x_n) \in
H^n(K,\mu_2^{\tensor n})$ to $L$ is nontrivial since its further
restriction to $k(E_1 \times \dotsm \times E_n)$ is nontrivial by
Theorem~\ref{thm:nontrivial}.  Hence
Proposition~\ref{prop:twisted_split} implies that $q$ is anisotropic
over $K$.
\end{proof}

\begin{proof}[Proof of Theorem~\ref{thm:main}]
By Bogomolov's trick (Corollary~\ref{cor:trick}), we find $x_1,
\dotsc, x_n \in K$ such that $K/k(x_1,\dotsc,x_n)$ has odd degree.
If $q$ is as in \eqref{eq:counter_example}, then by
Theorem~\ref{thm:locally}, $q$ is anisotropic yet locally isotropic
over $k(x_1,\dotsc,x_n)$.  Finally, by Proposition~\ref{prop:trick},
$q$ is a counterexample to the local-global principle for isotropy
over $K$.
\end{proof}

To give an explicit example, let $a,b,c \in \overline{\Q} \bslash
\{0,1\}$ be any algebraic integers all divisible by a common odd prime
ideal in a number field containing them. For example, take $a = b = c
= 3$.  Then over the function field $K = \C(x,y,z)$, the quadratic form
$$
q = \quadform{1,x,y,z,xy,xz,yz,(x-1)(y-1)(z-1)(x-a)(y-b)(z-c)}
$$
is isotropic over every completion $K_v$ associated to a 
discrete valuation $v$~on~$K$, and yet $q$ is anisotropic over $K$.  

\begin{remark}
\label{rem:curves}
Let $k$ be any algebraically closed field of characteristic $\neq 2$.
When $K/k$ is a finitely generated field of transcendence degree $1$,
then $K=k(X)$ for a smooth projective curve $X$ over $k$. Any binary
quadratic form $q$ over $K$ is similar to $\Pfister{a\!}\; =
\quadform{1,-a}$ for some $a \in K\mult$, and $q$ is isotropic if and
only if $a$ is a square.  For any discrete valuation $v$ on $K$, we
have that $q$ is isotropic over $K_v$ if and only if $a$ is a square
in $K_v$, equivalently (since $k$ is algebraically closed and
characteristic~$\neq 2$), $v(a)$ is even.  Thus if $q$ is locally
isotropic at all discrete valuations on $K$ then the divisor of the
rational function $a$ on $X$ can be written as~$2D$ for a divisor~$D$
on~$X$.  The divisor class of $D$ is 2-torsion in $\Pic(X)$ and it is
trivial if and only if $a$ is a square in $K$.  Conversely, if $X$
admits a nontrivial $2$-torsion element of $\Pic(X)$, then twice this
element is the divisor of a rational function $a \in K$ and the
local-global principle fails for $\Pfister{a}$. Thus the local-global
principle for isotropy fails for $K$ if and only if the Picard group
of $X$ admits a nontrivial $2$-torsion element, equivalently (again,
since $k$ is algebraically closed and characteristic $\neq 2$), the
genus of $X$ is positive.  Equivalently, the local-global principle
for isotropy holds for quadratic forms over $K$ if and only if $K/k$
is purely transcendental.

In fact, we see that Proposition~\ref{prop:twisted_split} (and hence Proposition~\ref{prop:local_isotropic}) is false
for $n=1$ by considering the trivial class in $H^1(K,\mu_2)$ and $d
\in K\mult$ any nonsquare.
\end{remark}

\begin{remark}
\label{rem:surfaces_finite_fields}
When $k$ is the algebraic closure of a finite field of characteristic
$\neq 2$, Theorem~\ref{thm:locally} still holds for $n=2$ assuming
that the elliptic curves $E_1$ and $E_2$ are not isogenous.  Indeed,
by Proposition~\ref{prop:main_unram}, the restriction of the symbol
$(x_1)\cdot (x_2) \in H^2(K,\mu_2^{\tensor 2})$ to $L$ is still
unramified, and the only thing left to verify is that it is
nontrivial.  We can check this by further restriction to $k(E_1 \times
E_2)$, where the symbol is the restriction to the generic point of a
class in $\Het^1(E_1,\mu_2) \otimes \Het^1(E_2,\mu_2)$
by~\S\ref{sec:kummer}.  However, standard computations of the Brauer
group of $E_1 \times E_2$, cf.\ \cite[\S3]{skoro_zarhin}, show that if
$E_1$ is not isogenous to $E_2$, then in fact $\Br(E_1 \times E_2)
\isom \Het^1(E_1,\mu_2)\otimes\Het^1(E_2,\mu_2)$, so that each such
cup product class is indeed nontrivial in the Brauer group.  Then, as
before, Proposition~\ref{prop:local_isotropic} implies that the
local-global principle for isotropy fails for $q$ as in
\eqref{eq:counter_example} over $K$, hence Theorem~\ref{thm:main} also
holds in this case.
\end{remark}

\section{A geometric presentation lemma}
\label{sec:general}

The method for producing locally isotropic but globally anisotropic
quadratic forms of dimension $2^n$ over function fields of
transcendence degree $n$ presented in this work is different from the
one employed in \cite[\S6]{APS} for $n=2$.  There, we first proved a
kind of geometric presentation lemma about the existence of nontrivial
unramified cohomology (in degree 2) over quadratic extensions.
Specifically, using Hodge theory, we proved
\cite[Proposition~6.4]{APS} that given any smooth projective surface
$S$ over an algebraically closed field of characteristic zero, there
exists a double cover $T \to S$ with $T$ smooth and $\Hur^2(k(T)/k,
\mu_2^{\tensor 2}) = \Br(T)[2] \neq 0$.  It has been an open question
ever since whether such a geometric presentation lemma holds for
unramified cohomology in higher degree.

\begin{conj}
\label{conj:double}
Let $K$ be a finitely generated field of transcendence degree $n$ over
an algebraically closed field $k$ of characteristic $\neq 2$.  Then
either $\Hur^n(K/k,\mu_2^{\tensor n}) \neq 0$ or there exists a
separable quadratic extension $L/K$ such that
$\Hur^n(L/k,\mu_2^{\tensor n}) \neq 0$.
\end{conj}

Assuming this conjecture, we can give a more direct proof of the
existence of quadratic forms representing a failure of the
local-global principle for isotropy without using the 
construction involving generalized Kummer varieties in \S\ref{sec:kummer}.

\begin{prop}
\label{prop:reduce}
Let $K$ be a finitely generated field of transcendence degree $n$ over
an algebraically closed field $k$ of characteristic $\neq 2$.
If Conjecture~\ref{conj:double} holds for~$K$, then the local-global
principle for isotropy of quadratic forms fails to hold in dimension
$2^n$ over $K$.
\end{prop}

Before proceeding with the proof of Proposition~\ref{prop:reduce}, we
recall a standard application of the Milnor conjectures for the Witt
group.  Since we could not find a suitable reference, we also provide
a proof.

\begin{lemma}
\label{lem:symbol}
Let $K$ be a field of characteristic $\neq 2$.  If $K$ is a
$C_n$-field then every element in $H^n(K,\mu_2^{\tensor n})$ is a
symbol.
\end{lemma}
\begin{proof}
By the Milnor conjectures for the Witt group, as proved by
Voevodsky~\cite{voevodsky:Milnor_conjecture}
and Orlov, Vishik, Voevodsky~\cite{orlov_vishik_voevodsky}, there
exists a surjective homomorphism $e_n : I^n(K) \to
H^n(K,\mu_2^{\tensor n})$ taking $n$-fold Pfister forms to symbols,
where $I^n(K)$ is the $n$th power of the fundamental ideal of the Witt
group of $K$.  Thus it suffices to prove that every element in
$I^n(K)$ is represented by a Pfister form.  Let $q$ be an anisotropic
quadratic form representing a class in $I^n(K)$.  By the
Arason--Pfister Hauptsatz (see
\cite[Ch.~4,~Theorem~5.6]{scharlau:book}), $q$ has dimension $\geq
2^n$, but since we are assuming that $K$ is a $C_n$-field, every
quadratic form of dimension $> 2^n$ is isotropic, hence $q$ has
dimension $2^n$.

Now we recall that every anisotropic form $q$ of dimension $2^n$ in
$I^n(K)$ is similar to a Pfister form over (any field) $K$, cf.\
\cite[Corollaire~4.3.7]{kahn}.  Indeed, let $K(q)$ be the function
field of the projective quadric defined by $q$.  Then $q_{K(q)} \in
I^n(K(q))$. Since $q$ is isotropic over $K(q)$, the anisotropic part
of $q_{K(q)}$ over $K(q)$ has dimension smaller than $2^n$, hence by
the Arason--Pfister Hauptsatz must be zero, thus $q$ is hyperbolic
over $K(q)$.  Being anisotropic over $K$ and hyperbolic over
$K(q)$, the quadratic form~$q$ is thus similar to a Pfister form over $K$,
see \cite[Ch.~4,~Theorem~5.4(i)]{scharlau:book}.

Since $K$ is assumed to be a $C_n$-field, and $I^{n+1}(K)$ is
additively generated by $(n+1)$-fold Pfister forms by
\cite[Ch.~4,~Lemma~5.5]{scharlau:book}, which are hyperbolic as soon
as they are isotropic by \cite[Ch.~4,~Corollary~1.5]{scharlau:book},
we conclude that $I^{n+1}(K)=0$.  We now argue that any quadratic
form in $I^n(K)$ that is similar to a Pfister form is actually a Pfister
form.  Indeed, if $\psi$ is any Pfister form in $I^n(K)$ and $a \in
K^\times$, then $\Pfister{\!a\!\!}\!  \tensor\, \psi = \psi \perp
-a\psi$ is in $I^{n+1}(K)=0$, hence $\psi \isom a\psi$. Thus our
anisotropic quadratic form $q$ in $I^n(K)$ is a Pfister form, proving
the desired statement.
\end{proof}

\vspace*{-1pt}
We do not know, in the spirit of
\cite[II.4.5~Remark~3]{serre:galois_cohomology} and
\cite{krashen_matzri}, whether the statement of Lemma~\ref{lem:symbol}
holds for Galois cohomology modulo $\ell$ for primes $\ell \neq 2$.

\begin{proof}[Proof of Proposition~\ref{prop:reduce}]
First, by Lemma~\ref{lem:symbol}, every element in
$H^n(K,\mu_2^{\tensor n})$ is a symbol since $K$ is a $C_n$-field by
Tsen--Lang theory~\cite{lang}.  Proposition~\ref{prop:local_isotropic}
(applied with $d=1$) implies that for any symbol $(a_1)\dotsm (a_n)$
in $\Hur^n(K/k,\mu_2^{\tensor n})$, the $n$-fold Pfister form
$\Pfister{a_1,\dotsc,a_n}$ is locally isotropic.  If we assume that
$\Hur^n(K/k,\mu_2^{\tensor n}) \neq 0$, then taking a nontrivial
unramified symbol $(a_1)\dotsm (a_n)$, the $n$-fold Pfister form
$\Pfister{a_1,\dotsc,a_n}$ is locally isotropic but is anisotropic by
Proposition~\ref{prop:twisted_split}, giving a counterexample to the
local-global principle for isotropy over $K$.

Now assume that $\Hur^n(K/k,\mu_2^{\tensor n})=0$ and that
$\Hur^n(L/k,\mu_2^{\tensor n}) \neq 0$ for some separable quadratic
extension ${L=K(\sqrt{d})}$ of $K$.  By Tsen--Lang theory (e.g.,
\cite[II.4.5]{serre:galois_cohomology}), $L$ is also a $C_n$-field,
hence by Lemma~\ref{lem:symbol} every element in $H^n(L,\mu_2^{\tensor
n})$ is a symbol.  Thus we can choose a nontrivial unramified symbol
$(a_1)\dotsm (a_n) \in \Hur^n(L/k,\mu_2^{\tensor n})$.  Since the
corestriction map $H^n(L,\mu_2^{\tensor n}) \to H^n(K,\mu_2^{\tensor
n})$ preserves unramified cohomology, and we have assumed that
$\Hur^n(K/k,\mu_2^{\tensor n})=0$, we see that the corestriction of
$(a_1)\dotsm (a_n)$ is trivial.    By the restriction-corestriction exact
sequence for Galois cohomology, see \cite[Satz~4.5]{arason} or \cite[I\S2~Exercise~2]{serre:galois_cohomology}, we have that $(a_1) \dotsm (a_n)$ is
in the image of the restriction map $H^n(K,\mu_2^{\tensor n}) \to
H^n(L,\mu_2^{\tensor n})$, and thus we can take $a_1, \dotsc, a_n \in
K\mult$.  Then by Proposition~\ref{prop:local_isotropic}, the twisted
Pfister form $\Pfister{a_1,\dotsc,a_n; d}$ is locally isotropic over $K$ but
globally anisotropic.
\end{proof}

However, under the hypothesis in which we prove
Theorem~\ref{thm:main}, namely, that $k$ is not the algebraic closure
of a finite field, our method allows us to prove
Conjecture~\ref{conj:double}.

\begin{theorem}
\label{thm:pres}
Let $k$ be an algebraically closed field of characteristic $\neq 2$.
If $k$ is not the algebraic closure of a finite field then
Conjecture~\ref{conj:double} holds for any finitely generated field
$K$ of transcendence degree $n$ over $k$.
\end{theorem}
\begin{proof}
By Bogomolov's trick (Corollary~\ref{cor:trick}) we consider $K$ as an
extension $K/K_0$ of odd degree over a rational function field $K_0 =
k(x_1,\dotsc,x_n)$. By Theorem~\ref{thm:nontrivial}, the symbol $(x_1)
\dotsm (x_n) \in H^n(K_0,\mu_2^{\tensor n})$ is nontrivial over the
(separable) quadratic extension $L_0 = K_0(\sqrt{f})$ for $f \in K$
defined by \eqref{eq:y}.  Since $K/K_0$ and $L/K_0$ have relatively
prime degree, $L= K \tensor_{K_0} L_0$ is a quadratic extension of $K$
and $L/L_0$ has odd degree.  Thus by a standard restriction-corestriction
argument, the symbol $(x_1) \dotsm (x_n)$ remains nontrivial when
restricted from $L_0$ to $L$.  By Proposition~\ref{prop:main_unram},
it is unramified over $L_0$, hence it remains unramified over $L$.
\end{proof}

\begin{remark}
When $k$ is the algebraic closure of a finite field of characteristic
$\neq 2$, then Conjecture~\ref{conj:double} holds for $n=2$.  Indeed,
following the proof of Theorem~\ref{thm:pres}, we only need to show
that $(x_1) \cdot (x_2) \in H^2(K_0,\mu_2^{\tensor n})$ is nontrivial
over the quadratic extension $L_0 = K_0(\sqrt{f})$, which follows from
Remark~\ref{rem:surfaces_finite_fields}.
\end{remark}

Thus we have reduced Conjecture~\ref{conj:double} to $k$ the algebraic
closure of a finite field.  However, the construction of nontrivial
higher degree unramified cohomology on varieties over a finite field
(or the algebraic closure of a finite field) is an open problem.  In
degree 3, this is related to the integral Tate conjecture.  Currently,
there are no known smooth projective threefolds over a finite field
with nontrivial unramified cohomology in degree 3; investigating this
is a favorite problem of Colliot-Th\'el\`ene, see
\cite[Question~5.4]{CT_kahn}.  The smallest known dimensions in which
such varieties exist is 5 (see \cite{pirutka}), and recently, 4 (see
\cite{scavia_suzuki}).  Of course, one wonders whether the cup product
class on a product of three elliptic curves, as in \S\ref{sec:kummer},
is nontrivial over a finite field.  One might also investigate the
same class on the associated generalized Kummer variety over a finite
field.

\providecommand{\bysame}{\leavevmode\hbox to3em{\hrulefill}\thinspace}
\providecommand{\href}[2]{#2}

\end{document}